\documentclass[preprint,12pt]{elsarticle}

\usepackage{amssymb}
\usepackage{amsthm}

\newcommand{\sysn}{\left\{\begin{array}{rcl}}
\newcommand{\sysk}{\end{array}\right.}

\newtheorem{theorem}{Theorem}[section]

\theoremstyle{example}
\newtheorem{example}[theorem]{Example}

\newtheorem{proposition}[theorem]{Proposition}
\theoremstyle{definition}
\newtheorem{definition}[theorem]{Definition}

\newtheorem{corollary}[theorem]{Corollary}

\begin{document}

\begin{frontmatter}



\title{On sequential separability of functional spaces}


\author[label1]{Alexander V. Osipov}

\ead{OAB@list.ru}

\tnotetext[label1]{The research has been supported by Act 211
Government of the Russian Federation, contract ¹ 02.A03.21.0006.}

\author[label2]{Evgenii G. Pytkeev}

 \ead{pyt@imm.uran.ru}

\tnotetext[label2]{The research has been supported by the Russian
Foundation for Basic Research (no.~16-01-00649, no.~16-07-00266).}

\address{Ural Federal
 University, Institute of Mathematics and Mechanics, Ural Branch of the Russian Academy of Sciences, 16,
 S.Kovalevskaja street, 620219, Ekaterinburg, Russia}

\begin{abstract}
In this paper, we give necessary and sufficient conditions for the
space $B_1(X)$ of first Baire class
 functions on a Tychonoff space $X$, with pointwise topology, to be (strongly) sequentially
 separable. Also we claim that there are spaces $X$ such that $B_1(X)$ is not sequentially
 separable space, but $C_p(X)$ is sequentially
 separable (the Sorgenfrey line, the Niemytzki plane).
\end{abstract}


\end{frontmatter}



\section{Introduction}

In \cite{vel}, \cite{glm} were given necessary and sufficient
conditions for the space $C_{p}(X)$ of continuous real-valued
functions on a space $X$, with pointwise topology, to be
sequentially separable. Also in \cite{glm} was given necessary and
sufficient condition for the space $C_{p}(X)$ to be strongly
sequentially separable.

In this paper, we give necessary and sufficient conditions for the
space $B_1(X)$ of first Baire class
 functions on a space $X$, with pointwise topology, to be sequentially separable and strongly sequentially separable.

\section{Main definitions and notation}

Throughout this article all topological spaces are considered
 Tychonoff. As usually, we will be denoted by $C_{p}(X)$ ($B_1(X)$) a
 set of all real-valued continuous functions $C(X)$ (a set of all first Baire class
 functions $B_1(X)$ i.e., pointwise limits of continuous functions) defined on
 $X$ provided with the pointwise convergence topology. If $X$ is a
 space and $A\subseteq X$, then the sequential closure of $A$,
 denoted by $[A]_{seq}$, is the set of all limits of sequences
 from $A$. A set $D\subseteq X$ is said to be sequentially dense
 if $X=[D]_{seq}$. If $D$ is a countable sequentially dense subset
 of $X$ then $X$ call sequentially separable space.

 Call $X$ strongly sequentially separable if $X$ is separable and
 every countable dense subset of $X$ is sequentially dense.

 We recall that a subset of $X$ that is the
 complete preimage of zero for a certain function from~$C(X)$ is called a zero-set.
A subset $O\subseteq X$  is called  a cozero-set (or functionally
open) of $X$ if $X\setminus O$ is a zero-set. If a set
$Z=\bigcup_{i} Z_i$ where $Z_i$ is a zero-set of $X$ for any $i\in
\omega$ then $Z$ is called $Z_{\sigma}$-set of $X$. Note that if a
space $X$ is a perfect normal space, then class of
$Z_{\sigma}$-sets of $X$ coincides with class of $F_{\sigma}$-sets
of $X$.

It is well known (\cite{kur}), that $f\in B_1(X)$ if and only if
$f^{-1}(G)$ --- $Z_{\sigma}$-set for any open set $G$ of real line
$\mathbb{R}$.

Further we use the following theorems.

\medskip

\begin{theorem} (\cite{vel}). \label{th31} A space $C_p(X)$ is
sequentially separable if and only if there
 exist  a condensation (one-to-one
continuous map) $f: X \mapsto Y$ from a space $X$ on a
 separable metric space $Y$, such that $f(U)$ --- $F_{\sigma}$-set
 of $Y$ for any cozero-set $U$ of $X$.
\end{theorem}
\medskip

\begin{theorem} (\cite{vel}). \label{th32} A space $B_1(X)$ is
sequentially separable for any separable metric space $X$.
\end{theorem}
\medskip

Note that proof this theorem  given more, namely there exist a
countable subset $S\subset C(X)$, such that $[S]_{seq}=B_1(X)$.

\section{Sequentially separable}

The main result of this paper is a next theorem.

\begin{theorem}\label{th1}
 A space $B_1(X)$ is sequentially separable if and only if there
 exist a bijection $\varphi: X \mapsto Y$ from a space $X$ onto a
 separable metrizable space $Y$, such that

\begin{enumerate}

\item $\varphi^{-1}(U)$ --- $Z_{\sigma}$-set of $X$ for any open
set $U$ of $Y$;

\item  $\varphi(T)$ --- $F_{\sigma}$-set of $Y$ for any zero-set
$T$ of $X$.

\end{enumerate}

\end{theorem}

\begin{proof} $(1) \Rightarrow (2)$. Let $B_1(X)$ be a sequentially
separable space, and $S$ be a countable sequentially dense subset
 of $B_1(X)$. Consider a topology $\tau$ generated by the family
 $\mathcal{P}=\{f^{-1}(G): G$ is an open set of $\mathbb{R}$ and
 $f\in S \}$. A space $Y=(X,\tau)$ is a separable metrizable space
 because $S$ is a countable dense subset
 of $B_1(X)$. Note that a function $f\in S$, considered as map from
 $Y$ to $\mathbb{R}$, is a continuous function. Let $\varphi$ be the identity
 map from $X$ on $Y$.

 We claim that $\varphi^{-1}(U)$ --- $Z_{\sigma}$-set of $X$ for any open
set $U$ of $Y$. Note that class of $Z_{\sigma}$-sets is closed
under a countable unions and finite intersections of its elements.
It follows that it is sufficient to prove for any $P\in
\mathcal{P}$. But $\varphi^{-1}(P)$ --- $Z_{\sigma}$-set for any
$P\in \mathcal{P}$ because $f\in S\subset B_1(X)$.

Let $T$ be a zero-set of $X$ and $h$ be a characteristic function
of $T$. Since $T$ is a zero-set of $X$, $h\in B_1(X)$. There are
$\{f_n\}_{n\in \omega}\subset S$ such that $\{f_n\}_{n\in
\omega}\mapsto h$. Since $S\subset C_p(Y)$, $h\in B_1(Y)$ and,
hence, $h^{-1}(\frac{1}{2},\frac{3}{2})=T$ is a $Z_{\sigma}$-set
of $Y$.

$(2) \Rightarrow (1)$. Let $\varphi$ be a bijection from $X$ on
$Y$ satisfying the conditions of theorem. Then $h=f\circ
\varphi\in B_1(X)$ for any $f\in C(Y)$
($h^{-1}(G)=\varphi^{-1}(f^{-1}(G))$ --- $Z_{\sigma}$-set of $X$
for any open set $G$ of $\mathbb{R}$). Moreover $g=f\circ
\varphi^{-1}\in B_1(Y)$ for any $f\in B_1(X)$ because of
$\varphi(Z)$ is a $Z_{\sigma}$-set of $Y$ for any a
$Z_{\sigma}$-set $Z$ of $X$. Define a map $F: B_1(X) \mapsto
B_1(Y)$ by $F(f)=f\circ \varphi^{-1}$. Since $\varphi$ is a
bijection, $C_p(Y)$ embeds in $F(B_1(X))$ i.e., $C(Y)\subset
F(B_1(X))$. By Theorem \ref{th32}, each subspace $D$ such that
$C(Y)\subset D\subset B_1(Y)$ is sequentially separable. Thus
$B_1(X)$ (homeomorphic to $F(B_1(X))$) is sequentially separable.
\end{proof}

\medskip

\begin{corollary} A space $B_1(X)$ is sequentially separable for
any regular space $X$ with a countable network.

\end{corollary}

\begin{proof} Let $X$ be a regular space $X$ with a countable
network. Then there are a countable network $\{F_n: n\in \omega\}$
of $X$ where $F_n$ is closed subset of $X$ for $n\in \omega$.
Consider a topology $\tau$ on $X$ generated by the family
 $\mathcal{P}=\{F_n, X\setminus F_n : n\in \omega\}$. Let
 $Z=(X,\tau)$. The space $Z$ is zero-dimensional space with countable
 base and, hence, it is metrizable space. Let $h: Z\mapsto X$ be the identity
 map. Note that any element of $\mathcal{P}$ is a $F_{\sigma}$-set of $X$.
 Hence $h(U)$ is a $F_{\sigma}$-set of $X$ for any open set $U$ of
 $Z$. Since $h$ is a condensation, $h^{-1}(T)$ is a zero-set of $Z$
 for any zero-set $T$ of $X$. Consider $\varphi=h^{-1}: X\mapsto
 Z$. Then $\varphi$ satisfies all the conditions of the Theorem \ref{th1}.

\end{proof}
In \cite{vel} was proved that identity
 maps of Sorgenfrey line $\mathcal{S}$ onto $\mathbb{R}$  and of
Niemytzki Plane $\mathcal{N}$ onto the closed upper half-plane
$\mathbb{R}^{2}_{+}$ satisfies the condition of the Theorem
\ref{th31}. Hence $C_p(\mathcal{S})$ and $C_p(\mathcal{N})$ are
sequentially separable spaces.

\medskip

We claim that $B_1(\mathcal{S})$ and $B_1(\mathcal{N})$ are not
sequentially separable spaces.

\medskip
We recall some concepts and facts related to the space of first
Baire class  functions.

$\bullet$ A map $f:X \mapsto Y$ be called $Z_{\sigma}$-map, if
$f^{-1}(Z)$ is a $Z_{\sigma}$-set of $X$ for any zero-set $Z$ of
$Y$.

Note that a continuous map is a $Z_{\sigma}$-map.

$\bullet$ A space $X$ be called an analytic space if there is a
continuous map $\varphi$ from $\mathbb{P}$ (space of irrational
numbers) onto $X$.

$\bullet$ A space $X$ be called an $K$-analytic space if there are
$\breve{C}$ech-complete Lindel$\ddot{e}$of space $Y$ and a
continuous map $f$ from $Y$ onto $X$.

An analytic space or a compact space are $K$-analytic space
(\cite{kur}).

If in definition of $K$-analytic space the continuous map is
replace by $Z_{\sigma}$-map we get a wider class of spaces ---
class of $K_{\sigma}$-analytic spaces (\cite{pyt}).

Note that if $X$ is a $K_{\sigma}$-analytic space then $X^{n}$ is
 Lindel$\ddot{e}$of space for each $n\in \omega$ (\cite{pyt}).

\begin{example} There is
a space $X$ such that $C_p(X)$ is sequentially separable space but
$B_1(X)$ is not sequentially separable.
\end{example}

\begin{proof}
We claim that $B_1(\mathcal{S})$ is not sequentially separable
space.

Consider $B_1(\mathcal{S})$. Let $p: \mathcal{S} \mapsto
\mathbb{R}$ be the identity
 map. Then $p(T)$ is a $F_{\sigma}$-set of $\mathbb{R}$ for any cozero-set
 $T$ of $S$ (\cite{vel}).

 Suppose that $B_1(\mathcal{S})$ is sequentially separable space.
 By Theorem \ref{th1}, there are separable metrizable space $Y$
 and a bijection $\varphi: \mathcal{S} \mapsto Y$ from $\mathcal{S}$ onto
 $Y$ such that $\varphi(T)$ is a $F_{\sigma}$-set of $Y$ for any
 closed subset $T$ of $\mathcal{S}$ and $\varphi^{-1}(U)$ is a  $F_{\sigma}$-set of $\mathcal{S}$ for any
 open subset $U$ of $Y$. Consider the map $\varphi\circ p^{-1}:
 \mathbb{R} \mapsto Y$. Since $\varphi\circ p^{-1}$ is a Borel
 function, the space $Y$ is an analytic separable metrizable space
 (\cite{kur}). Thus the map $\varphi^{-1}: Y\mapsto \mathcal{S}$ is a
 $Z_{\sigma}$-bijection of analytic separable metrizable space
 $Y$. Then $\mathcal{S}$ is a $K_{\sigma}$-analytic space and,
 hence, $\mathcal{S}^2$ is a Lindel$\ddot{e}$of space, a
 contradiction.

\end{proof}

The same proof remains valid for the space $B_1(\mathcal{N})$.

\section{Strongly sequentially separable}

From \cite{ps}, we note that $B_1(X)$ is separable if and only if
$X$ has a coarser second countable topology.

In \cite{glm} be characterized those spaces $X$ so that $C_p(X)$
is strongly sequentially separable. The following definition and
theorems are relevant. For a proof of Theorem \ref{th2} see
\cite{gn}, Theorem \ref{th3} see \cite{glm}, and more information
on the property $\gamma$ see \cite{gn}.

\begin{definition} A family $\alpha$ of subsets of $X$ is called
an $\omega$-cover of $X$ if for every finite $F\subset X$ there is
a $U\in \alpha$ such that $F\subset U$.
\end{definition}

\begin{theorem}\label{th2}(\cite{gn}). The following are equivalent:

\begin{enumerate}
\item $C_p(X)$ is Frechet-Urysohn;

\item $X$ has the property $\gamma$: for any open $\omega$-cover
$\alpha$ of $X$ there is a sequence $\beta\subset \alpha$ such
that $\lim\inf \beta=X$.

\end{enumerate}

\end{theorem}

\begin{theorem}\label{th3}(\cite{glm}). The space $C_p(X)$ is strongly sequentially separable if and only if
$X$ has a coarser second countable topology, and every coarser
second countable topology for $X$ has the property $\gamma$.
\end{theorem}

We characterize those spaces $X$ so that $B_1(X)$ is strongly
sequentially separable.

\begin{theorem}\label{th4} The function
space $B_1(X)$ is strongly sequentially separable if and only if
$X$ has a coarser second countable topology, and for any bijection
$\varphi$ from a space $X$ onto a
 separable metrizable space $Y$, such that  $\varphi^{-1}(U)$ --- $Z_{\sigma}$-set of $X$ for any open set
$U$ of $Y$, the space $Y$ has the property $\gamma$.
\end{theorem}

\begin{proof}
$(\Rightarrow)$. Assume that $B_1(X)$ is strongly sequentially
separable. Let $\varphi$ be a  bijection from a space $X$ onto a
 separable metrizable space $Y$, such that  $\varphi^{-1}(U)$ --- $Z_{\sigma}$-set of $X$ for any open set
$U$ of $Y$. Then for any $f\in C(Y)$, $h=f\circ \varphi\in
B_1(X)$. Since $\varphi$ is a bijection, a map $F: C(Y) \mapsto
B_1(X)$ (define as $F(f)=f\circ \varphi$) is an embeds in $B_1(X)$
i.e., $F(C_p(Y))\subset B_1(X)$. Note that $F(C_p(Y))$ is dense
separable subset of $B_1(X)$ and, hence, $C_p(Y)$ is strongly
sequentially separable. By Theorem \ref{th3}, the space $Y$ has
the property $\gamma$.

$(\Leftarrow)$. Assume that $X$ has a coarser second countable
topology. Then $B_1(X)$ is separable (Theorem 1 in  \cite{ps}).

Let $A$ be a countable dense subset of $B_1(X)$. We wish to show
that for any $f\in B_1(X)$ there is some sequence $\{f_i:i\in
\omega\}\subset A$ such that $f$ is the limit of the sequence.

Consider a topology $\tau$ generated by the family

 $\mathcal{P}=\{g^{-1}(G): G$ is an open set of $\mathbb{R}$ and
 $g\in A\bigcup\{f\} \}$. A space $Y=(X,\tau)$ is a separable metrizable space
 because $A$ is a countable dense subset
 of $B_1(X)$. Note that a function $g\in A\bigcup\{f\}$, considered as map from
 $Y$ to $\mathbb{R}$, is a continuous function. Let $\varphi$ be the identity
 map from $X$ on $Y$.

 We claim that $\varphi^{-1}(U)$ --- $Z_{\sigma}$-set of $X$ for any open
set $U$ of $Y$. Note that class of $Z_{\sigma}$-sets is closed
under a countable unions and finite intersections of its elements.
It follows that it is sufficient to prove for any $P\in
\mathcal{P}$. But $\varphi^{-1}(P)$ --- $Z_{\sigma}$-set for any
$P\in \mathcal{P}$ because $g\in A\bigcup\{f\}\subset B_1(X)$.
 It follows that the space $Y$ has
the property $\gamma$ and, hence, $C_p(Y)$ is strongly
sequentially separable. Since $A$ is countable dense subset
 of $B_1(X)$, the set $A$ (where a map in $A$ considered as map from
 $Y$ to $\mathbb{R}$) is a countable dense subset
 of $C_p(Y)$. Hence there is some sequence $\{f_i:i\in
\omega\}\subset A$ such that $f$ is the limit of the sequence.

\end{proof}

It is consistent and independent for arbitrary $X$, that $C_p(X)$
is strongly sequentially separable if and only if $C_{p}(X)$ is
Frechet-Urysohn.

Note that the theorem \ref{th2} implies the following proposition
(Corollary 17 in \cite{glm}).

\begin{proposition}(Cons (ZFC)) The following are equivalent:

\begin{enumerate}

\item  $C_p(X)$ is strongly sequentially separable;

\item   $X$ is countable;

\item  $C_p(X)$ is separable and Frechet-Urysohn.

\end{enumerate}

\end{proposition}

For a space $B_1(X)$ we have a following

\begin{corollary}(Cons (ZFC)) The following are equivalent:

\begin{enumerate}

\item  $B_1(X)$ is strongly sequentially separable;

\item   $X$ is countable;

\item  $B_1(X)$ is separable and Frechet-Urysohn.

\end{enumerate}

\end{corollary}

\medskip

Recall that the cardinal $\mathfrak{p}$ is the smallest cardinal
so that there is a collection of $\mathfrak{p}$ many subsets of
the natural numbers with the strong finite intersection property
but no infinite pseudo-intersection. Note that $\omega_1 \leq
\mathfrak{p} \leq \mathfrak{c}$. (See \cite{dou} for more on small
cardinals including $\mathfrak{p}$.)

\begin{example}  Assume that $\omega_1<\mathfrak{p}$. There is
uncountable space $X$ such that $B_1(X)$ is strongly sequentially
separable but not Frechet-Urysohn.
\end{example}

\begin{proof}
Let $X$ be $\omega_1$ with the discrete topology. Clearly, that
$X$ is separable submetrizable space and so $B_1(X)$ is a
separable, dense subspace of $\mathbb{R}^{X}$. We know (from
Theorem 11 in \cite{glm}) that $\mathbb{R}^{X}$ must be strongly
sequentially separable space and so $B_1(X)$ is strongly
sequentially separable. Note that $\mathbb{R}^{\omega_1}$ is not
Frechet-Urysohn.
\end{proof}

\section{Acknowledgement}

This work was supported by Act 211 Government of the Russian
Federation, contract ¹ 02.A03.21.0006.

\bibliographystyle{model1a-num-names}
\bibliography{<your-bib-database>}







\end{document}